\documentclass[12pt]{article}%
\usepackage{amssymb, amsfonts, amsmath}
\usepackage{mathrsfs}
\usepackage{graphicx}
\usepackage[usenames,dvipsnames]{xcolor}
\usepackage{tikz}
\usepackage{float}
\usepackage{enumerate}
\usepackage{comment}
\usepackage{setspace}
\usepackage{caption}
\usepackage{subcaption}
\usepackage{comment}
\usepackage[hidelinks]{hyperref}
\usepackage{array}%
\usepackage{amsmath}%
\usepackage{amsfonts}%
\usepackage{amssymb}
\providecommand{\U}[1]{\protect\rule{.1in}{.1in}}
\newcolumntype{C}[1]{>{\centering\let\newline\\\arraybackslash\hspace{0pt}}m{#1}}
\providecommand{\U}[1]{\protect\rule{.1in}{.1in}}

\newtheorem{theorem}{Theorem}

\newtheorem{coro}[theorem]{Corollary}

\newtheorem{lemma}[theorem]{Lemma}

\newenvironment{proof}[1][Proof]{\noindent\textbf{#1.} }{\ \hfill \rule{0.5em}{0.5em}}
\newcounter{example}

\providecommand{\rad}{\operatorname{rad}}

\providecommand{\gid}{\gamma^{ID}}
\providecommand{\ggid}{G(\gamma^{ID})}
\providecommand{\leqn}{1 \leq i \leq n}
\providecommand{\gpr}{\gamma_{pr}}
\providecommand{\gt}{\gamma_{t}}
\providecommand{\bull}{\mathcal{B}}

\providecommand{\ir}{\operatorname{ir}}
\definecolor{lgray}{gray}{0.95}
\definecolor{mgray}{gray}{0.40}
\usetikzlibrary{fit,positioning}
\tikzstyle{std}=[ circle, draw=black,fill=black,thick, inner sep=2pt, minimum size=2.5mm]
\tikzstyle{wstd}=[ circle, draw=black,fill=white,thick, inner sep=2pt, minimum size=2.5mm]
\tikzstyle{ir}=[ circle, draw=black,fill=green,thick,  inner sep=2pt, minimum size=2mm]
\tikzstyle{mp}=[circle, draw=black,fill=Dandelion,thick,  inner sep=2pt, minimum size=2mm]
\tikzstyle{bred}=[circle, draw=black,fill=red,thick,  inner sep=2pt, minimum size=2mm]
\tikzstyle{smred}=[ circle, draw=black,fill=red,thick,  inner sep=1pt, minimum size=1.5mm]
\tikzstyle{bblue}=[ circle, draw=black,fill=blue,thick,  inner sep=1pt, minimum size=1.5mm]
\tikzstyle{regRed}=[ circle, draw=black,fill=red,thick,  inner sep=2pt, minimum size=2.5mm]
\tikzstyle{sqRed}=[ rectangle, draw=black,fill=red,thick,  inner sep=2pt, minimum size=2.5mm]
\tikzstyle{tr}=[color=black, style=dotted]
\tikzstyle{sp}=[color=ProcessBlue, line width = 2pt]
\begin{document}

\title{\textbf{A note on some variations of the $\gamma$-graph}}
\author{C. M. Mynhardt\thanks{Supported by the Natural Sciences and Engineering
Research Council of Canada.} {\ }and L. E. Teshima\\Department of Mathematics and Statistics\\University of Victoria, P. O. Box 3045, Victoria, BC\\\textsc{Canada} V8W 3P4\\{\small kieka@uvic.ca, lteshima@uvic.ca}}
\maketitle

\begin{abstract}
For a graph $G$, the \emph{$\gamma$-graph of $G$}, $G(\gamma)$, is the graph
whose vertices correspond to the minimum dominating sets of $G$, and where two
vertices of $G(\gamma)$ are adjacent if and only if their corresponding
dominating sets in $G$ differ by exactly two adjacent vertices. In this paper,
we present several variations of the $\gamma$-graph including those using
identifying codes, locating-domination, total-domination, paired-domination,
and the upper-domination number. For each, we show that for any graph $H$,
there exist infinitely many graphs whose $\gamma$-graph variant is isomorphic
to $H$.

\end{abstract}

\vspace{-5mm}

\begin{center}
\emph{In memory of Peter Slater, September 30, 1946 - September 27, 2016.}
\end{center}

\noindent\textbf{Keywords:\hspace{0.1in}} $\gamma$-Graphs, Domination,
Reconfiguration problems

\noindent\textbf{AMS Subject Classification 2010:\hspace{0.1in}} 05C69



Given a graph $G=(V,E)$, a set $S \subseteq V$ is said to be a
\emph{dominating set} of $G$ if for each $v\in V$, $v$ is either in $S$ or
adjacent to a vertex in $S$. The minimum cardinality of a dominating set is
the \emph{dominating number} $\gamma(G)$, and a set is a $\gamma$-\emph{set}
if it is dominating and has cardinality $\gamma(G)$. The \emph{private
neighbourhood} of a vertex $v$ with respect to a vertex set $S$ is the set
$pn[v,S] = N[v]-N[S-\{v\}]$; therefore, a dominating set $S$ is minimal
dominating if for each $u \in S$, $pn[u,S]$ is nonempty. A set $S\subseteq
V(G)$ is \emph{irredundant} if $\operatorname{pn}(v,S)\neq\varnothing$ for
each $v\in S$, and \emph{maximal irredundant }if $S$ is irredundant but no
proper superset of $S$ is irredundant. The \emph{irredundance number
}$\operatorname{ir}(G)$ is the minimum cardinality of a maximal irredundant
set of $G$. For a review of domination principles, see \cite{HHS1,HHS2}. In
general, we follow the notation of \cite{CL}.

First defined by Fricke et al. in 2011 \cite{FHHH11}, the \emph{$\gamma$-graph
of a graph $G$} is the graph $G(\gamma) = (V(G(\gamma)),E(G(\gamma)))$, where
each vertex $v\in V(G(\gamma))$ corresponds to a $\gamma$-set $S_{v}$ of $G$.
The vertices $u$ and $v$ in $G(\gamma)$ are adjacent if and only if there
exist vertices $u^{\prime}$ and $v^{\prime}$ in $G$ such that $u^{\prime
}v^{\prime}\in E(G)$ and $S_{v} = (S_{u} - u^{\prime})\cup\{v^{\prime}\}$.
This model of adjacency is referred to as the \emph{slide adjacency} or
sometimes simply as the \emph{slide-model}. For additional results on $\gamma
$-graphs, see \cite{B15, CHH10, EdThesis, FHHH11}.

An initial question of Fricke et al. was to determine exactly which graphs are
$\gamma$-graphs \cite{FHHH11}; they showed that every tree is the $\gamma
$-graph of some graph, and further conjectured that every graph is the
$\gamma$-graph of some graph. Later that year, Connelly et al. \cite{CHH10}
proved this conjecture to be true.

\begin{theorem}
\label{thm:CCH11:gammaGraph} \emph{\cite{CHH10}}  For any graph $H$, there
exists some graph $G$ such that $G(\gamma) \simeq H$.  That is, every graph is
the $\gamma$-graph of some graph. 
\end{theorem}

Subramanian and Sridharan \cite{SS08} independently defined a different
$\gamma$-graph of a graph $G$, denoted $\gamma\cdot G$. The vertex set of
$\gamma\cdot G$ is the same as $G(\gamma)$; however, for $u,w \in
V(\gamma\cdot G)$ with associated $\gamma$-sets $S_{u}$ and $S_{w}$ in $G$,
$u$ and $w$ are adjacent in $\gamma\cdot G$ if and only if there exist some
$v_{u} \in S_{u}$ and $v_{w} \in S_{w}$ such that $S_{w}= (S_{u} - \{v_{u}\})
\cup\{v_{w}\}$. This version of the $\gamma$-graph was dubbed the ``single
vertex replacement adjacency model" by Edwards \cite{EdThesis}, and is
sometimes more colloquially referred to as the ``jump model" and the ``jump
$\gamma$-graph". Further results concerning $\gamma\cdot G$ can be found in
\cite{LV10, SS09, SS13}. Notably, if $G$ is a tree or a unicyclic graph, then
there exists a graph $H$ such that $\gamma\cdot H = G$ \cite{SS09}.
Conversely, if $G$ is the (jump) $\gamma$-graph of some graph $H$, then $G$
does not contain any induced $K_{3,2}$, $P_{3} \vee K_{2}$, or $(K_{2}\cup
K_{1}) \vee2 K_{1}$ \cite{LV10}.

Also using a jump-adjacency model, Haas and Seyffarth \cite{HS14} define the
\emph{$k$-dominating graph} of $G$, $D_{k}(G)$, as the graph with vertices
corresponding to the $k$-dominating sets of $G$ (i.e. the dominating sets of
cardinality at most $k$). Two vertices in the $k$-dominating graph are
adjacent if and only if the symmetric difference of their associated
$k$-dominating sets contains exactly one element. Additional results can be
found in \cite{HS17, HIMNOST16, SMN16}.

In this paper, we examine several variations to the $\gamma$-graph, and
provide realizability results similar to Theorem \ref{thm:CCH11:gammaGraph}.
For consistency, all figures show the construction required to realize the
graph $H = K_{4}-e$ (the 2-fan). Unless otherwise specified, we consider only
the slide-adjacency model in our variations.


\section{The $\ir$, $\gamma_{t}$, $\gamma_{pr}$ and $\gamma_{c}$-graphs}

To begin, we examine four domination-related parameters and their respective
extensions of Theorem \ref{thm:CCH11:gammaGraph}. A vertex set $S$ is said to
\emph{totally-dominate} a graph $G$ if it is dominating and for each $v\in S$
there exists $u\in S$ such that $u$ and $v$ are adjacent. The cardinality of a
smallest total-dominating set is the \emph{total-domination number}
$\gamma_{t}(G)$, as first introduced by Cockayne et al. in \cite{CDH80}.
Closely related, a \emph{paired-dominating} set $S$ is a total-dominating with
the additional requirement that the induced subgraph $G[S]$ has a perfect
matching. The \emph{paired-domination number} $\gamma_{pr}$ was defined
similarly by Haynes and Slater in \cite{HS98}. Since every paired-dominating
set is also a total-dominating set, for every graph $G$ without isolated
vertices, $\gamma(G)\leq\gt(G)\leq\gpr(G)$ \cite{HHS1}. A
\emph{connected-dominating} set $S$ is a dominating set where $G[S]$ is
connected, and the cardinality of a smallest connected-dominating set is the
\emph{connected-domination number} $\gamma_{c}$, as defined by Sampathkumar
and Walikar \cite{SW79}. For all nontrivial connected graphs $G$,
$\ir(G)\leq\gamma(G)\leq\gamma_{c}(G)\leq\gamma_{t}(G)$ and $\gamma
(G)\leq2\ir(G)-1$ \cite{HHS1}.

We say the $\ir$\emph{-}$\emph{graph}$, the \emph{$\gamma_{t}$-graph}, the
\emph{$\gamma_{pr}$-graph}, and the \emph{$\gamma_{c}$-graph} of $G$ are the
graphs with vertices representing the minimum-cardinality maximal irredundant,
total-dominating, paired-dominating, and connected dominating sets of $G$,
respectively, and where adjacency is defined using the vertex-slide model.
Using the same construction as Connelly et al. in \cite{CHH10} and restated
below, we find a result analogous to Theorem \ref{thm:CCH11:gammaGraph} for
the $\ir$-graph, $\gamma_{t}$-graph, $\gamma_{pr}$-graph, and $\gamma_{c}$-graph.

\begin{coro}
\label{thm:prGraph} Every graph $H$ is the $\ir$-graph, $\gamma_{pr}$-graph,
$\gamma_{t}$-graph, and $\gamma_{c}$-graph of infinitely many graphs.
\end{coro}

\noindent For reference, we restate Connelly et al.'s construction for a graph
$G$ with $G(\gamma)\simeq H$.

\noindent\underline{Construction:} Given some graph $H$ with $V(H) =
\{v_{1},v_{2},\dots, v_{n}\}$, to construct a graph $G$ with $G(\gamma)\simeq
H$, begin with a copy of $H$ and attach vertices $a,b,c$ to every vertex of
$H$. Then, add two (or more) pendant vertices to $c$, labelled as $c_{1}$ and
$c_{2}$ (see Figure \ref{fig:gprConst}).


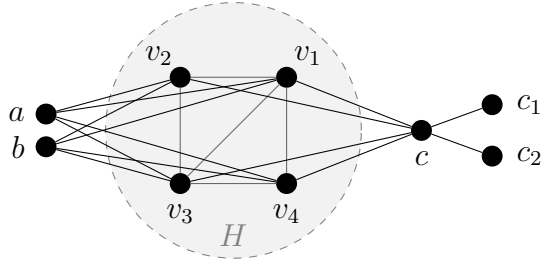
\begin{figure}[H]
\centering
\begin{tikzpicture}					
							\fill[lgray, dashed] (0,0) circle (1.7cm);
							
							\node [std](v1) at (45:1cm) {};
								\path (v1) ++(45:4 mm) node (x1L) {$v_1$};
							\node [std](v2) at (135:1cm){};
									\path (v2) ++(135:4 mm) node (x2L) {$v_2$};
							\node [std](v3) at (225:1cm) {};
								\path (v3) ++(270:4 mm) node (x3L) {$v_3$};
							\node [std](v4) at (315:1cm){};
									\path (v4) ++(270:4 mm) node (x4L) {$v_4$};
							\draw[mgray](v1)--(v2)--(v3)--(v4)--(v1)--(v3);
							
							\draw[gray, dashed] (0,0) circle (1.7cm);
							\node (H) at (0,-1.4) [color=gray](L){$H$};	
	
							\node [std](a) at (175:25mm)[label=left:$a$]  {};			
							\node [std](b) at (185:25mm)[label=left:$b$]  {};

							\node [std](c) at (0:25mm)[label=below:$c$]  {};	
							
							\foreach \i in {1,2,3,4}
									{\draw(a)--(v\i)--(b);
									\draw(c)--(v\i);}
							
						\foreach \i / \j in {1/20, 2/-20}
							{ \path (c) ++(\j:10 mm) node (c\i) [std] {};
								\path (c\i) ++(0:5 mm) node (cL\i)  {$c_{\i}$};	
								\draw(c)--(c\i);
							}
							
				\end{tikzpicture}
\caption{The graph $G$ constructed from $H$ from \cite{CHH10}.}%
\label{fig:gprConst}%
\end{figure}

From the pendant vertices, $c$ is in every $\gamma$-set of $G$; however, as
$a$ and $b$ remain undominated, $\{c\}$ is not itself a $\gamma$-set. Thus,
$\gpr(G) \geq\gamma(G) \geq2$ (likewise $\gamma_{c}(G) \geq\gt(G) \geq
\gamma(G) \geq2$). It follows that for each $v_{i}\in V(H)$, the set
$S_{i}=\{c,v_{i}\}$ is a $\gamma$-set. Since $\ir(G) \leq2 = \gamma(G) \leq2
\ir(G)-1$, it follows that $S_{i}$ is an $\ir$-set. Moreover, since $cv_{i}
\in E(G)$, each $S_{i}$ is also a $\gpr$-set, a $\gt$-set, and $\gamma_{c}%
$-set. Since neither $\{c,a\}$ nor $\{ c,b\}$ is dominating, the collection
$\{S_{i}: 1 \leq i \leq n\}$ consists of all the $\ir$, $\gamma$, $\gt$,
$\gpr$, $\gamma_{c}$-sets of $G$.


\section{The $\gid$-graph}

\label{subsec:GIDExist}

A popular variation on domination is the topic of identifying codes. A vertex
set $S$ is an \emph{identifying code} (ID-code) if for each $v \in V$, the
closed neighbourhood of $v$ and $S$ have unique, nonempty intersection. The
size of a smallest ID-code is denoted $\gid(G)$, and an ID-code with
cardinality $\gid(G)$ is called an $\gid$-\emph{set}. If $\gid(G)$ is finite,
$G$ is said to be \emph{identifiable} (or \emph{distinguishable}); otherwise,
if $G$ is not identifiable, $\gid(G)$ is defined to be $\gid(G)=\infty$. The
\emph{intersection set} of a vertex $v$ with respect to a vertex subset $S$ is
the set $I_{S}(v) = N[v]\cap S$. Thus, $S$ is an identifying code of $G$ if
all of its intersection sets are unique and nonempty.

Originally introduced by Karpovsky et al. in 1998 \cite{KCL98}, ID-codes were
proposed as a model for the positioning of fault-detection units on
multiprocessor systems (for additional references, see Lobstein's extensive
bibliography \cite{Biblio}). Consider now the problem of migrating the
detecting units from one configuration to another, such that only one
detecting unit can be moved at time to an adjacent processor, and at each step
the configuration remains an ID-code. When given a certain starting
configuration, what other configurations are reachable under these conditions?
How many steps are required to move between them? Are there multiple routes
from start to destination, or are we stuck with a single path? To aid in
addressing this family of questions, we define the $\gid$-\emph{graph of a
graph $G$}, $G(\gid)=(V(\gid),E(\gid))$, similarly to the $\gamma$-graph, but
where the vertices now correspond to the $\gid$-sets in $G$ instead.

As a first result, we extend Theorem \ref{thm:CCH11:gammaGraph} to
$\gid$-graphs. The construction and proof are similar; however, in
consideration of the additional identification requirements, multiple copies
of the graph $\mathcal{C} = C_{4}\circleddash K_{1}$, the \emph{depleted
corona of $C_{4}$}, in Figure \ref{fig:gidHand} are used to force certain
vertices into the $\gid$-set. Given any graph $G^{\prime}$, we construct the
graph $G$ by adding an edge between $x_{1} \in V(\mathcal{C})$ and any $v\in
V(G^{\prime})$ (we say $\mathcal{C}$ is \emph{attached to $G^{\prime}$ at $v$}).

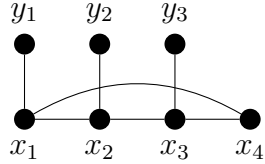
\begin{figure}[tbh]
\centering
\begin{tikzpicture}				
					
							\foreach \i/\x in {1/1.5, 2/2.5,3/3.5,4/4.5}
								\node [std](x\i) at (\x,1.5)[label=below:$x_\i$] {};	
							
							\draw(x1)--(x2)--(x3)--(x4);		
							
							\foreach \i/\x in {1/1.5, 2/2.5,3/3.5}
								\node [std](y\i) at (\x,2.5)[label=above:$y_\i$] {};							
								
							\foreach \i/\x in {1/1.5, 2/2.5,3/3.5}
								\draw (y\i)--(x\i);
										
							\draw(x1) to [out=30,in=150](x4);								
						
				\end{tikzpicture}
\caption{The graph $\mathbf{\mathcal{C}}$ in Lemmas \ref{lem:C} and
\ref{lem:gidHand}.}%
\label{fig:gidHand}%
\end{figure}

\begin{lemma}
\label{lem:C}  The vertex set $X=\{x_{1},x_{2},x_{3}\}$ is the unique
$\gid$-set of $\mathcal{C}$.
\end{lemma}


\begin{proof}
Suppose that $\mathcal{C}$ has a $\gid$-set $S$. Since the vertices in
$Y=\{y_{1},y_{2},y_{3}\}$ are all pendant vertices in $\mathcal{C}$, for each
$1 \leq i \leq3$, either $x_{i}$ or $y_{i}$ is in $S$. Since $X$ is
identifying, it follows that $\gid(\mathcal{C}) = 3$ and that $X$ is a
$\gid$-set of $\mathcal{C}$.

Notice that since $|S|=3$ and each $y_{i}$ is pendant, $S\subseteq X\cup Y$
and $x_{4}\notin S$. To dominate $x_{4}$, either $x_{1}$ or $x_{3}$ is in $S$.
Without loss of generality, say $x_{1} \in S$. To show uniqueness, we need
only verify that the sets $S_{2}=\{x_{1},y_{2},x_{3}\}$, $S_{3}=\{x_{1}%
,x_{2},y_{3}\}$ and $S_{2,3}=\{x_{1},y_{2},y_{3}\}$ are not identifying. For
$S_{2}$, $I_{S_{2}}(x_{3})= I_{S_{2}}(y_{3})=\{x_{3}\}$ and is therefore not
identifying. Likewise for $S_{2,3}$, $I_{S_{2,3}}(x_{3})= I_{S_{2,3}}%
(y_{3})=\{x_{3}\}$. Finally for $S_{3}$, $I_{S_{3}}(x_{4})= I_{S_{3}}%
(y_{1})=\{x_{1}\}$. It follows that $S=X =\{x_{1},x_{2},x_{3}\}$ is the unique
$\gid$-set of $\mathcal{C}$.
\end{proof}


\begin{lemma}
\label{lem:gidHand}  Let $G^{\prime}$ be any graph, and construct $G$ by
attaching $\mathcal{C}$ to $G^{\prime}$ at some $v\in V(G^{\prime})$. If $S$
is any $\gid$-set of $G$, then $\{x_{1},x_{2},x_{3}\} \subseteq S$. Moreover,
if $S^{\prime}$ is a $\gid$-set of $G^{\prime}$, then $S^{\prime}\cup
\{x_{1},x_{2},x_{3}\}$ is identifying in $G$.
\end{lemma}


\begin{proof}
Let $V(G^{\prime}) = \{v_{1},v_{2},\dots,v_{n}\}$, and suppose that $G$ was
constructed by attaching $\mathcal{C}$ to $G^{\prime}$ at $v_{n}$. Suppose
that $G$ has an $\gid$-set $S$. Regardless of whether $v_{n} $ is in $S$ or
not, each vertex of $Y=\{y_{1},y_{2},y_{3}\}$ remains pendant and so for each
$1 \leq i \leq3$, either $x_{i}$ or $y_{i}$ is in $S$. Using the same
arguments as in Lemma \ref{lem:C}, $\{x_{1},x_{2},x_{3}\} \subseteq S$.

Now suppose that $S^{\prime}$ is a $\gid$-set of $G^{\prime}$ and consider
$S=S^{\prime}\cup\{x_{1},x_{2},x_{3}\}$ in $G$. Since $S^{\prime}$ is
identifying, for each $v_{i} \in V(G^{\prime})$, the intersection set
$I_{S^{\prime}}(v_{i})$ in $G^{\prime}$ is unique. In $G$, all intersection
sets of $v_{1}, v_{2}, \dots, v_{n-1}$ remain the same. The intersection set
of $v_{n}$ becomes $I_{S}(v_{n}) = I_{S^{\prime}}(v_{n}) \cup\{x_{1}\} \neq
I_{s}(y_{1})$. From the first portion of this lemma, the sets $I_{S}(x_{i})$
for $1 \leq i \leq4$ are also unique. It follows that $S$ is identifying in
$G$.
\end{proof}


Notice that the converse of the second portion of Lemma \ref{lem:gidHand} does
not hold. For a counterexample, consider the graph $G$ in Figure
\ref{fig:leCE} constructed by attaching $\mathcal{C}$ to a copy of $C_{4}$
with $V(C_{4}) =\{v_{1},v_{2},v_{3},v_{4}\}$. Although $\{v_{1},v_{3}%
,x_{1},x_{2},x_{3}\}$ is a $\gid$-set of $G$, $\{v_{1},v_{3}\}$ is not a
$\gid$-set of $C_{4}$.

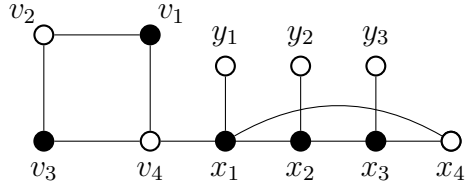
\begin{figure}[H]
\centering
\begin{tikzpicture}					
							\node [std](v1) at (45:1cm) {};
								\path (v1) ++(45:4 mm) node (x1L) {$v_1$};
							\node [wstd](v2) at (135:1cm){};
									\path (v2) ++(135:4 mm) node (x2L) {$v_2$};
							\node [std](v3) at (225:1cm) {};
								\path (v3) ++(270:4 mm) node (x3L) {$v_3$};
							\node [wstd](v4) at (315:1cm){};
									\path (v4) ++(270:4 mm) node (x4L) {$v_4$};
							\draw(v1)--(v2)--(v3)--(v4)--(v1);								
										
							\foreach \label/\rad in {1/10,2/20,3/30,4/40}
								\path (v4) ++(0:\rad mm) node (x\label) [std]  {};
							\draw(v4)--(x4);
							
							\foreach \i in {1, 2,3,4}
								\path (x\i) ++(-90:4 mm) node (xL\i) {$x_{\i}$};				
					
						\foreach \i/\x in {1,2,3}
								{	\path (x\i) ++(90:10 mm) node (y\i) [wstd]  {};
									\draw (y\i)--(x\i);			
									\path (y\i) ++(90:4 mm) node (yL\i) {$y_{\i}$};	
								}
										
							\draw(x1) to [out=30,in=150](x4);			
								\node [wstd](x4w) at (x4){};										
							
				\end{tikzpicture}
\caption{Counterexample to the converse of Lemma \ref{lem:gidHand}.}%
\label{fig:leCE}%
\end{figure}

\begin{theorem}
\label{thm:gamma-ID}  Every graph $H$ is the $\gid$-graph of some graph.
\end{theorem}


\begin{proof}
Let $H = (V(H),E(H))$ be any nonempty graph with $V(H) = \{v_{1},v_{2}%
,\dots,v_{n}\}$. We construct a new graph $G$ such that $G(\gid) \simeq H$.

\noindent\underline{Construction:} Begin with a copy of $H$ and for each
$v_{i} \in V(H)$, attach two copies of the graph $\mathcal{C}$ from Figure
\ref{fig:gidHand} to $H$ at $v_{i}$, labelled as $\mathcal{C}_{i}$ and
$\mathcal{C}_{i}^{*}$ with $V(\mathcal{C}_{i}) = \{x_{i,1},x_{i,2}%
,x_{i,3},x_{i,4}, y_{i,1}, y_{i,2}, y_{i,3}\}$ and $V(\mathcal{C}_{i}^{*}) =
\{x_{i,1}^{*},x_{i,2}^{*},x_{i,3}^{*},x_{i,4}^{*}, y_{i,1}^{*}, y_{i,2}^{*},
y_{i,3}^{*}\}$. Now, add vertices $a$ and $b$ so that $av_{i} \in E(G)$ and
$bv_{i} \in E(G)$ for all $i=1,2,\dots, n$. Finally, attach two more copies of
$\mathcal{C}$, $\mathcal{C}_{a}$ and $\mathcal{C}_{b}$, at $a$ and $b$,
respectively (see Figure \ref{fig:gidConst}).

\noindent Consider the vertex set,
\begin{align*}
X = \left( \bigcup\limits_{\substack{1 \leq j \leq n \\1 \leq k \leq3}}
\{x_{j,k}, x_{j,k}^{*}\}\right)  \cup\left( \bigcup\limits_{1\leq k \leq3}
\{x_{a,k},x_{b,k}\}\right) .
\end{align*}
In particular, notice that $X$ consists of all the vertices within the various
$\mathcal{C}$ graphs that Lemma \ref{lem:gidHand} demonstrates are in every
$\gid$-set of $G$.

\begin{figure}
\centering
\begin{tikzpicture}					
							\fill[lgray, dashed] (0,0) circle (1.7cm);
							
							\node [std](v1) at (45:1cm)[label=above:$v_1$] {};
							\node [std](v2) at (135:1cm)[label=above:$v_2$] {};
							\node [std](v3) at (225:1cm)[label=below:$v_3$] {};
							\node [std](v4) at (315:1cm)[label=below:$v_4$] {};
							\draw[mgray](v1)--(v2)--(v3)--(v4)--(v1)--(v3);
							
							\draw[gray, dashed] (0,0) circle (1.7cm);
							\node (H) at (0,-1.4) [color=gray](L){$H$};			
												
							\node (x31) at (215:25mm) [label=above:$x_{3,1}$] {};
							\node (x41) at (305:25mm) [label=left:$x_{4,1}$] {};
							
							\draw(v3)--(x31);
							\draw(v4)--(x41);
							\node (x31*) at (235:25mm)[label=right:$x_{3,1}^*$] {};
							\node (x41*) at (325:25mm) [label=above:$x_{4,1}^*$]{};	
							\draw(v3)--(x31*);
							\draw(v4)--(x41*);
						
							\path (v1) ++(25:20 mm) node (x11) [std]  {};
								\draw(v1)--(x11);
								
							\foreach \label/\rad in {2/8,3/16,4/24}
								\path (x11) ++(25:\rad mm) node (x1\label) [std]  {};
							\draw(x11)--(x14);
							
							\foreach \x/\i in {1/2, 2/2,3/2,4/2}
								\path (x1\x) ++(115:3 mm) node (x1L\x) {$x_{1,\x}$};										
							
							\foreach \x/\y in {1/1, 2/2, 3/3}
								\path (x1\x) ++(-65:6mm) node (y1\y) [std]  {};
							\foreach \x/\y in {1/1, 2/2,3/3}
							\draw(x1\x)--(y1\y);	
							
							\foreach \y/\i in {1/2, 2/2,3/2}
								\path (y1\y) ++(-65:4 mm) node (y1L\y) {$y_{1,\y}$};								
							
							\draw(x11) to [out=0,in=240](x14);	
							
							\path (v1) ++(65:20 mm) node (x11*) [std]  {};
								\draw(v1)--(x11*);
								
							\foreach \label/\rad in {2/8,3/16,4/24}
								\path (x11*) ++(65:\rad mm) node (x1\label*) [std]  {};
							\draw(x11*)--(x14*);							
							
							\foreach \x/\i in {1/2, 2/2,3/2,4/2}
								\path (x1\x*) ++(-25:5 mm) node (x1L\x*) {$x_{1,\x}^*$};									
							
							\foreach \x/\y in {1/1, 2/2, 3/3}
								\path (x1\x*) ++(155:6mm) node (y1\y*) [std]  {};
							\foreach \x/\y in {1/1, 2/2,3/3}
							\draw(x1\x*)--(y1\y*);				
							
							\foreach \y/\i in {1/2, 2/2,3/2}
								\path (y1\y*) ++(155:5 mm) node (y1L\y*) {$y_{1,\y}^*$};										
						
							\draw(x11*) to [out=90,in=220](x14*);
							\path (v2) ++(115:20 mm) node (x21) [std]  {};
								\draw(v2)--(x21);
								
							\foreach \label/\rad in {2/8,3/16,4/24}
								\path (x21) ++(115:\rad mm) node (x2\label) [std]  {};
							\draw(x21)--(x24);
							
							\foreach \x/\i in {1/2, 2/2,3/2,4/2}
								\path (x2\x) ++(205:5 mm) node (x2L\x) {$x_{2,\x}$};										
							
							\foreach \x/\y in {1/1, 2/2, 3/3}
								\path (x2\x) ++(25:6mm) node (y2\y) [std]  {};
							\foreach \x/\y in {1/1, 2/2,3/3}
							\draw(x2\x)--(y2\y);	
							
							\foreach \y/\i in {1/2, 2/2,3/2}
								\path (y2\y) ++(25:5 mm) node (y2L\y) {$y_{2,\y}$};								
							
							\draw(x21) to [out=90,in=330](x24);	
							
							\path (v2) ++(155:20 mm) node (x21*) [std]  {};
								\draw(v2)--(x21*);
								
							\foreach \label/\rad in {2/8,3/16,4/24}
								\path (x21*) ++(155:\rad mm) node (x2\label*) [std]  {};
							\draw(x21*)--(x24*);							
							
							\foreach \x/\i in {1/2, 2/2,3/2,4/2}
								\path (x2\x*) ++(65:4 mm) node (x2L\x*) {$x_{2,\x}^*$};									
							
							\foreach \x/\y in {1/1, 2/2, 3/3}
								\path (x2\x*) ++(245:6mm) node (y2\y*) [std]  {};
							\foreach \x/\y in {1/1, 2/2,3/3}
							\draw(x2\x*)--(y2\y*);				
							
							\foreach \y/\i in {1/2, 2/2,3/2}
								\path (y2\y*) ++(245:4 mm) node (y2L\y*) {$y_{2,\y}^*$};										
						
							\draw(x21*) to [out=180,in=300](x24*);	
		
							\node [std](a) at (180:25mm)[label=below:$a$]  {};			
							
							\foreach \i in {1,2,3,4}
								\draw(a)--(v\i);
													
							\path (a) ++(180:15 mm) node (xa1) [std]  {};
								\draw(a)--(xa1);
								
							\foreach \label/\rad in {2/8,3/16,4/24}
								\path (xa1) ++(180:\rad mm) node (xa\label) [std]  {};
							\draw(xa1)--(xa4);
							
							\foreach \x in {1, 2,3,4}
								\path (xa\x) ++(90:3 mm) node (xaL\x) {$x_{a,\x}$};										
							
							\foreach \x/\y in {1/1, 2/2, 3/3}
								\path (xa\x) ++(270:6mm) node (ya\y) [std]  {};
							\foreach \x/\y in {1/1, 2/2,3/3}
							\draw(xa\x)--(ya\y);	
							
							\foreach \y/\i in {1/2, 2/2,3/2}
								\path (ya\y) ++(270:4 mm) node (yaL\y) {$y_{a,\y}$};								
							
							\draw(xa1) to [out=210,in=-30](xa4);	
						
							\node [std](b) at (0:25mm)[label=below:$b$]  {};					
							
							\foreach \i in {1,2,3,4}
								\draw(b)--(v\i);
															
							\draw(b)--(v1);
							\draw(b)--(v2);
							\draw(b)--(v3);
							\draw(b)--(v4);
					
							\path (b) ++(0:15 mm) node (xb1) [std]  {};
								\draw(b)--(xb1);
								
							\foreach \label/\rad in {2/8,3/16,4/24}
								\path (xb1) ++(0:\rad mm) node (xb\label) [std]  {};
							\draw(xb1)--(xb4);
							
							\foreach \x in {1, 2,3,4}
								\path (xb\x) ++(90:3 mm) node (xbL\x) {$x_{b,\x}$};										
							
							\foreach \x/\y in {1/1, 2/2, 3/3}
								\path (xb\x) ++(270:6mm) node (yb\y) [std]  {};
							\foreach \x/\y in {1/1, 2/2,3/3}
							\draw(xb\x)--(yb\y);	
							
							\foreach \y/\i in {1/2, 2/2,3/2}
								\path (yb\y) ++(270:4 mm) node (ybL\y) {$y_{b,\y}$};								
							
							\draw(xb1) to [out=-30,in=210](xb4);														
							
					\path (x31) ++(215:4 mm) node (x31C){};
					\fill[lgray] (x31C) circle (4mm);
						\draw[gray] (x31C) circle (4mm){};
						\node[std] at (x31){};
						\node at  (x31C) [color=gray]{$\mathcal{C}_3$};	
												
					\path (x31*) ++(235:4 mm) node (x31C*){};
					\fill[lgray] (x31C*) circle (4mm);
						\draw[gray] (x31C*) circle (4mm){};
						\node[std] at (x31*){};
						\node at  (x31C*) [color=gray]{$\mathcal{C}_3^*$};							
					
					\path (x41) ++(305:4 mm) node (x41C){};
					\fill[lgray] (x41C) circle (4mm);
						\draw[gray] (x41C) circle (4mm){};
						\node[std] at (x41){};
						\node at  (x41C)[color=gray]{$\mathcal{C}_4$};
						
					\path (x41*) ++(325:4 mm) node (x41C*){};
					\fill[lgray] (x41C*) circle (4mm);
						\draw[gray] (x41C*) circle (4mm){};
						\node[std] at (x41*){};
						\node at  (x41C*)[color=gray]{$\mathcal{C}_4^*$};
													
				\end{tikzpicture}
\caption{The graph $G$ constructed from $H$ in Theorem \ref{thm:gamma-ID}. }%
\label{fig:gidConst}%
\end{figure}
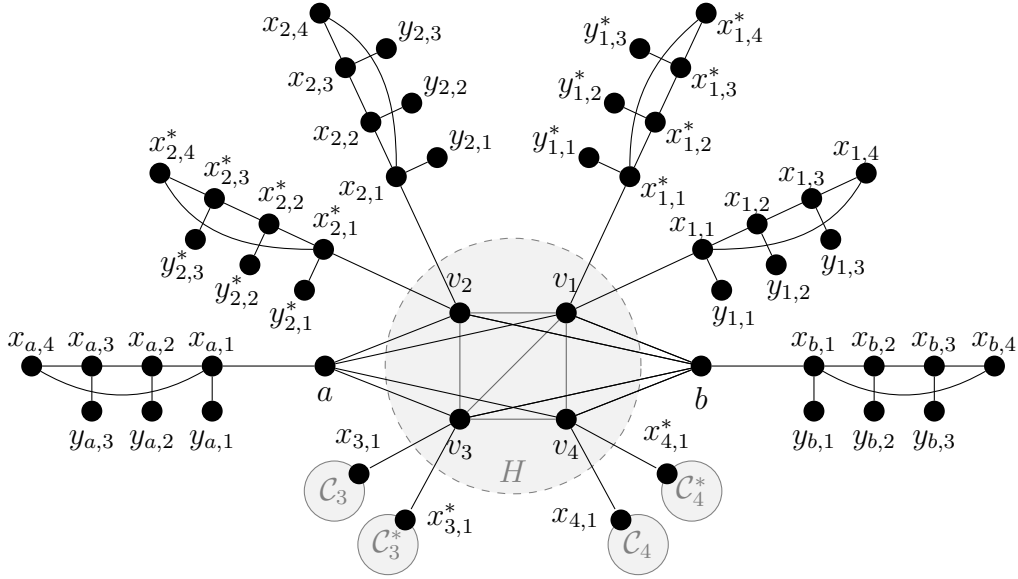

We first show that for each $1 \leq i \leq n$, the vertex set $S_{i} =
\{v_{i}\}\cup X$ is a $\gid$-set of $G$. From the construction of $G$, it is
clear that $S_{i}$ dominates $G$. Furthermore, from Lemma \ref{lem:gidHand},
we know that each vertex within a $\mathcal{C}$ subgraph is identified by
$S_{i}$. For vertices within $H$, the identifying set $I_{S_{i}}(v_{j})$ of
the vertex $v_{j}$ contains the unique pair $\{x_{j,1},x_{j,1}^{*}\}$,
ensuring that each $v_{j}$ is identified in $S_{i}$. Finally, $a$ and $b$ also
have the unique identifying sets $I_{S_{i}}(a)=\{x_{a,1}, v_{i}\}$ and
$I_{S_{i}}(b) = \{x_{b,1}, v_{i}\}$, respectively. It follows that $S_{i}$ is
identifying and dominating in $G$.

We now show that $S_{i}$ is a $\gid$-set. Notice that there are $2n+6$ pendant
vertices in $G$, and so $\gamma(G) \geq2n+6$. Indeed, it is easy to see that
$X=S_{i}-\{v_{i}\}$ is dominating in $G$, so $\gamma(G) = 2n+6$. Moreover,
$\gid(G) \geq\gamma(G) = 2n+6$. Again, by Lemma \ref{lem:gidHand}, we know
that any $\gid$-set of $G$ contains all of $X$; however $X$ is not identifying
as $I_{X}(a) = \{x_{a,1}\} = I_{X}(y_{a,1})$, and so, $\gid(G) \geq(2n+6)+1$.
Since $|S_{i}|=2n+7$ and $S_{i}$ is identifying, it follows that $S_{i}$ is a
$\gid$-set for all $\leqn$.

Let $\mathcal{S}_{\gid}=\{S_{1},S_{2},\dots,S_{n}\}$. We claim that
$\mathcal{S}_{\gid}$ is the collection of all $\gid$-sets of $G$. Since we
have already established that $\gid(G)=2n+7$ and that in every $\gid$-set,
$2n+6$ of the vertices are from $X$, every $\gid$-set of $G$ can viewed as
\textquotedblleft$X$-plus-one\textquotedblright. However, there is no single
vertex $w\in V(G)-V(H)$ such that $X\cup\{w\}$ identifies both pairs
$a,y_{a,1}$ and $b,y_{b,1}$. The set $S_{a}=\{a\}\cup X$ with $|S_{a}|=2n+7$
has $I_{S_{a}}(b)=\{x_{b,1}\}=I_{S_{a}}(y_{b,1})$ and is therefore not
identifying. Similarly, $S_{b}=\{b\}\cup X$ is not identifying. Thus,
$\mathcal{S}_{\gid}$ is the collection of all $\gid$-sets.

Consider now $G(\gid) = (V(\ggid),E(\ggid))$. By the above arguments,
$V(\ggid)$ $= \{v_{1}^{\prime},v_{2}^{\prime},\dots,v_{n}^{\prime}\}$, where
the vertex $v_{i}^{\prime}$ corresponds to the set $S_{i} \in\mathcal{S}%
_{\gid}$ for each $\leqn$. Since $v_{i}^{\prime}$ and $v_{j}^{\prime}$ in
$V(\ggid)$ are adjacent in $\ggid$ if and only if there exist $w_{i}\in S_{i}$
and $w_{j} \in S_{j}$ with $w_{i}w_{j}\in E(G)$ such that $S_{i} = (S_{j}
-\{w_{j}\}) \cup\{w_{i}\}$ and $S_{i}$ and $S_{j}$ differ at exactly one
vertex (that is, $v_{i}$ versus $v_{j}$), it follows that $v_{i}^{\prime}%
v_{j}^{\prime}\in E(\gid)$ if and only if $v_{i}v_{j} \in E(G)$. Therefore,
$\ggid \simeq H$ as required.
\end{proof}


In the construction of the graph $G$ in the proof of Theorem
\ref{thm:gamma-ID}, if instead of attaching only two copies of the graph
$\mathcal{C}$ to each vertex of the graph $H$, we attached three or more
copies, the same $\gid$-graph $H$ is obtained. This leads immediately to the
following corollary.

\begin{coro}
Every graph $H$ is the $\gid$-graph of infinitely many graphs.
\end{coro}

\section{The $\gamma_{L}$ and $\gamma_{t}^{L}$-graphs}

Introduced by Slater in 1988 \cite{Slater88}, a \emph{locating-dominating set}
$S$ of graph $G=(V,E)$ is a dominating set such that for each $v\in V-S$, the
set $N[v] \cap S$ is unique. In contrast to identifying codes,
locating-dominating sets do not require that the vertices of the dominating
set have unique neighbourhood intersection with the dominating set itself. The
minimum cardinality of a locating-dominating set, denoted by $\gamma_{L}(G)$,
is the \emph{locating-dominating number} of a graph $G$. A \emph{$\gamma_{L}%
$-set} of a graph is a minimum locating-dominating vertex subset. Since all
identifying codes are also locating-dominating sets, it follows that
$\gamma_{L} \leq\gid$ for all graphs. We reuse the notation of the
intersection set of a vertex $v$ and set $S$ from ID-codes; however, for
locating-domination, the sets $I_{S}(v)$ need only be unique for $v\notin S$.

We define the \emph{$\gamma_{L}$-graph of a graph $G$}, $G(\gamma
_{L})=(V(\gamma_{L}),E(\gamma_{L}))$, to the be graph where the vertex set
$V(\gamma_{L})$ is the collection of $\gamma_{L}$-sets of $G$. As with the
$\gamma$-graph, $u,w \in V(\gamma_{L})$ associated with $\gamma_{L}$-sets
$S_{u}$ and $S_{w}$ are adjacent in $G(\gamma_{L})$ if and only if there exist
$v_{u} \in S_{u}$ and $v_{w} \in S_{w}$ with $v_{u}v_{w}\in E(G)$, such that
$S_{u} = (S_{w} - \{v_{w}\}) \cup\{v_{u}\}$.

Given the similarities between ID-codes and locating-dominating sets, it is
not surprising that a similar result to Theorem \ref{thm:gamma-ID} exists for
$\gamma_{L}$-graphs.

\begin{theorem}
\label{thm:gammaL} Every graph $H$ is the $\gamma_{L}$-graph of infinitely
many graphs.
\end{theorem}

Since $\mathcal{C}$ does not have a unique $\gamma_{L}$-set (for example,
$\{x_{1},x_{2},x_{3}\}$ and $\{x_{1},x_{2},y_{3}\}$ are $\gamma_{L}$-sets), we
cannot use it in the construction to prove Theorem \ref{thm:gammaL}. Instead,
we use the very similar \emph{Bull graph}, $\bull$, as pictured in Figure
\ref{fig:gLhand}. Notice that $S=\{x_{1},x_{2}\}$ is a $\gamma_{L}$-set in
$\bull$. Moreover, since $S_{1}=\{x_{1}, y_{2}\}$ and $S_{2}=\{y_{1},x_{2}\}$
give $I_{S_{1}}(y_{1})=\{x_{2}\}=I_{S_{1}}(x_{3})$, and $I_{S_{2}}%
(y_{2})=\{x_{2}\}=I_{S_{2}}(x_{3})$, $S$ is the only $\gamma_{L}$-set of
$\bull$. The proof to Theorem \ref{thm:gammaL} then proceeds identically to
that of Theorem \ref{thm:gamma-ID}, substituting the use of $\bull$ for
$\mathcal{C}$.

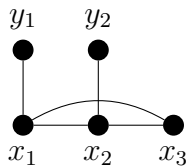
\begin{figure}[H]
\centering
\begin{tikzpicture}				
					
							\foreach \i/\x in {1/1.5, 2/2.5,3/3.5}
								\node [std](x\i) at (\x,1.5)[label=below:$x_\i$] {};	
							
							\draw(x1)--(x2)--(x3);		
							
							\foreach \i/\x in {1/1.5, 2/2.5}
								\node [std](y\i) at (\x,2.5)[label=above:$y_\i$] {};							
								
							\foreach \i/\x in {1/1.5, 2/2.5}
								\draw (y\i)--(x\i);
										
							\draw(x1) to [out=30,in=150](x3);								
						
				\end{tikzpicture}
\caption{The Bull graph $\bull$ used in the construction of Theorem
\ref{thm:gammaL}.}%
\label{fig:gLhand}%
\end{figure}

A variant of locating-domination, a vertex subset $S$ is a
\emph{locating-total dominating set} (LTDS) of a graph $G$ if $S$ a
locating-dominating set and if each vertex in $V(G)$ is adjacent to some
vertex in $S$. The \emph{locating-total domination number} $\gamma_{t}^{L}(G)$
is the minimum cardinality of a LTDS \cite{HHH06}. We define the $\gamma
_{t}^{L}$-graph of a graph $G$ analogously to the $\gamma_{L}$-graph. Since in
the construction of Theorem \ref{thm:gammaL}, there were no independent
vertices in the $\gamma_{L}$-sets, the following corollary is immediate.

\begin{coro}
Any graph $H$ is the $\gamma_{t}^{L}$-graph of infinitely many graphs.
\end{coro}


\section{The $\Gamma$-graph}

The final domination parameter we examine is $\Gamma(G)$, the
\emph{upper-domination number} of a graph $G$, defined to be the cardinality
of a largest minimal dominating set. The \emph{$\Gamma$-graph of a graph $G$}
and its associated parameters are defined analogously to $G(\gamma)$. Once
again, this variation requires the use of a new gadget: the graph
$\mathcal{Z}$ in Figure \ref{fig:Z}.

\begin{figure}[H]
\centering
\begin{tikzpicture}								
								
						\node [std](z) at (0,0)[label=below:$z$]  {};						
									
						\path (z) ++(0:10 mm) node (x2) [std] {};
						\draw(z)--(x2);	
														
						\foreach \i / \j in {1/90, 3/-90}
							{ \path (x2) ++(\j:8 mm) node (x\i) [std] {};
								\draw(z)--(x\i);
							}							
						
						\draw(x1)--(x3);	
						\draw(x1) to [out=-120,in=120](x3);		
							
						\foreach \i / \j  in {1/30, 2/30, 3/30}
							{
								\path (x\i) ++(0:15 mm) node (y\i) [std] {};
								\path (y\i) ++(\j:5 mm) node (yL\i)  {$y_{\i}$};
								\path (x\i) ++(\j:5 mm) node (xL\i)  {$x_{\i}$};		
								\draw(x\i)--(y\i);									
							}	
							
						\draw(y1)--(y3);	
						\draw(y1) to [out=-120,in=120](y3);		
							
				\end{tikzpicture}
\caption{The graph $\mathcal{Z}$ used in Theorem \ref{thm:bigGamma}.}%
\label{fig:Z}%
\end{figure}
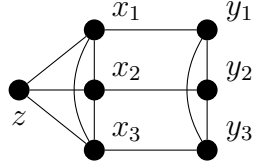

\begin{theorem}
\label{thm:bigGamma} Every graph $H$ is the $\Gamma$-graph of infinitely many graphs.
\end{theorem}


\begin{proof}
\underline{Construction:} The construction of a graph $G$ with $G(\Gamma
)\simeq H$ is similar to the previous results. Begin with a copy of the graph
$H$ with $V(H)=\{v_{1}, v_{2}, \dots, v_{n}\}$, and to each $v_{i}$, attach a
copy of the graph $\mathcal{Z}$ in Figure \ref{fig:Z} labelled $\mathcal{Z}%
_{i}$ with $V(Z_{i})=\{z_{i},x_{i,1},x_{i,2},x_{i,3}, y_{i,1}, y_{i,2},
y_{i,3}\}$ at vertex $x_{i,1}$ to $v_{i}$. Attach a final copy of
$\mathcal{Z}$ labelled $\mathcal{Z^{*}}$ ($V(\mathcal{Z}^{*})=\{z^{*}%
,x_{1}^{*},x_{2}^{*},x_{3}^{*}, y_{1}^{*}, y_{2}^{*}, y_{3}^{*}\}$) by joining
each $v_{i}$ to $z^{*}$.

\noindent For reference, we define $X_{i} = \{x_{i,1}, x_{i,2}, x_{i,3}\}$,
$Y_{i} = \{y_{i,1}, y_{i,2}, y_{i,3}\}$, $X^{*} = \{x_{1}^{*}, x_{2}^{*},
x_{3}^{*}\}$, and $Y^{*} = \{y_{1}^{*}, y_{2}^{*}, y_{3}^{*}\}$.

We claim that the $\Gamma$-sets of $G$ are $S_{1}, S_{2}, \dots, S_{n}$ where
for each $1 \leq i \leq n$,
\begin{align}
S_{i} = \{v_{i}\} \cup\left( \bigcup_{1\leq j \leq n} X_{j} \right)  \cup
Y^{*}.
\end{align}

\begin{figure}
\centering
\begin{tikzpicture}					
							\fill[lgray, dashed] (0,0) circle (1.7cm);
							
							\node [std](v1) at (45:1cm) {};
								\path (v1) ++(100:4 mm) node (x1L) {$v_1$};
							\node [std](v2) at (135:1cm){};
									\path (v2) ++(80:4 mm) node (x2L) {$v_2$};
							\node [std](v3) at (225:1cm) {};
								\path (v3) ++(280:4 mm) node (x3L) {$v_3$};
							\node [std](v4) at (315:1cm){};
									\path (v4) ++(260:4 mm) node (x4L) {$v_4$};
							\draw[mgray](v1)--(v2)--(v3)--(v4)--(v1)--(v3);
							
							\draw[gray, dashed] (0,0) circle (1.7cm);
							\node (H) at (0,-1.4) [color=gray](L){$H$};	
				\path (v3) ++(225:15 mm) node (x31) [std][label=right:$x_{3,1}$] {};								
				\path (v4) ++(-45:15 mm) node (x41) [std][label=left:$x_{4,1}$] {};		
							
							\draw(v3)--(x31);
							\draw(v4)--(x41);
							
					\path (x31) ++(225:5 mm) node (x31C){};
					\fill[lgray] (x31C) circle (5mm);
						\draw[gray] (x31C) circle (5mm){};
						\node[std] at (x31){};
						\node at  (x31C) [color=gray]{$\mathcal{Z}_3$};							
					
					\path (x41) ++(-45:5 mm) node (x41C){};
					\fill[lgray] (x41C) circle (5mm);
						\draw[gray] (x41C) circle (5mm){};
						\node[std] at (x41){};
						\node at  (x41C)[color=gray]{$\mathcal{Z}_4$};
						
						\path (v1) ++(60:15 mm) node (x11) [std] {};
						
						\foreach \i / \j in {2/8, 3/16}
							{ \path (x11) ++(0: \j mm) node (x1\i) [std] {};
								\draw(x11)--(x1\i);								
							}
					
						\path (x12) ++(-90:8 mm) node (z1) [std] [label=below:$z_1$] {};
						\draw (x11)--(z1)--(x13);
						\draw(z1)--(x12);															
							
						\draw(x11)--(x13);	
						\draw(x11) to [out=-30,in=210](x13);		
							
						\foreach \i / \j  in {1/45, 2/45, 3/45}
							{
								\path (x1\i) ++(90:15 mm) node (y1\i) [std] {};
								\path (y1\i) ++(\j:5 mm) node (yL1\i)  {$y_{1,\i}$};							
								\draw(x1\i)--(y1\i);									
							}	
							
						\foreach \i / \j  in {1/45, 2/45, 3/45}	
							\path (x1\i) ++(\j:5 mm) node (xL1\i)  {$x_{1,\i}$};							
							
						\draw(y11)--(y13);	
						\draw(y11) to [out=-30,in=210](y13);						
						
						\draw(v1)--(x11);

						
						\path (v2) ++(120:15 mm) node (x21) [std] {};
						
						\foreach \i / \j in {2/8, 3/16}
							{ \path (x21) ++(180: \j mm) node (x2\i) [std] {};
								\draw(x21)--(x2\i);								
							}
					
						\path (x22) ++(-90:8 mm) node (z2) [std] [label=below:$z_2$] {};
						\draw (x21)--(z2)--(x23);
						\draw(z2)--(x22);															
							
						\draw(x21)--(x23);	
						\draw(x21) to [out=210,in=-30](x23);		
							
						\foreach \i / \j  in {1/135, 2/135, 3/135}
							{
								\path (x2\i) ++(90:15 mm) node (y2\i) [std] {};
								\path (y2\i) ++(\j:5 mm) node (yL2\i)  {$y_{2,\i}$};							
								\draw(x2\i)--(y2\i);									
							}	
							
						\foreach \i / \j  in {1/135, 2/135, 3/135}	
							\path (x2\i) ++(\j:5 mm) node (xL2\i)  {$x_{2,\i}$};							
							
						\draw(y21)--(y23);	
						\draw(y21) to [out=210,in=-30](y23);		
						\draw(v2)--(x21);							
					
				
							\node [std](z) at (0:25mm)[label=below:$z^*$]  {};		
								\foreach \i in {1, 2, 3, 4}
									\draw(z)--(v\i);								
									
						\path (z) ++(0:10 mm) node (x2) [std] {};
						\draw(z)--(x2);	
														
						\foreach \i / \j in {1/90, 3/-90}
							{ \path (x2) ++(\j:8 mm) node (x\i) [std] {};
								\draw(z)--(x\i);
							}							
						
						\draw(x1)--(x3);	
						\draw(x1) to [out=-120,in=120](x3);		
							
						\foreach \i / \j  in {1/30, 2/30, 3/30}
							{
								\path (x\i) ++(0:15 mm) node (y\i) [std] {};
								\path (y\i) ++(\j:5 mm) node (yL\i)  {$y_{\i}^*$};
								\path (x\i) ++(\j:5 mm) node (xL\i)  {$x_{\i}^*$};		
								\draw(x\i)--(y\i);									
							}	
							
						\draw(y1)--(y3);	
						\draw(y1) to [out=-120,in=120](y3);						
							
				\end{tikzpicture}
\caption{The graph $G$ constructed from $H$ in Theorem \ref{thm:bigGamma}.}%
\label{fig:bigGamma}%
\end{figure}
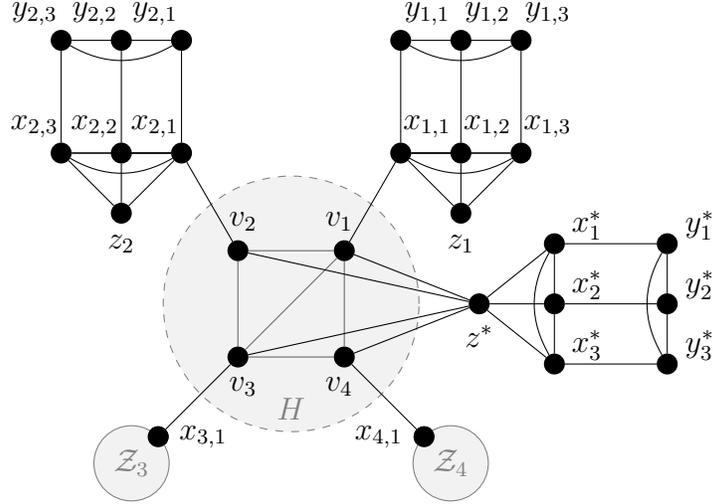

\noindent To begin, notice that $S_{i}$ is minimal dominating with $|S_{i}| =
3(n+1)+1=3n+4$; for each $1 \leq j \leq3$, $p[x_{i,j},S_{i}] =\{x_{i,j}\}$,
$pn[y_{j}^{*},S_{i}] = \{x_{1}^{*}\}$, and $pn[v_{i},S_{i}] = \{z^{*}\}$. We
proceed with a series of claims to demonstrate that the collection
$\mathcal{S}=\{S_{1}, S_{2},\dots,S_{n}\}$ contains the only $\Gamma$-sets of
$G$.

\begin{enumerate}
[(i)] \setlength{\parindent}{1pt} \setlength{\parskip}{1pt}

\item \emph{If $S$ is a minimal dominating set, then for each $1 \leq i \leq
n$ and $1 \leq j \leq3$, $\{z_{i}, x_{i,j}\} \not \subseteq S$. Likewise,
$\{z^{*}, x_{j}^{*}\} \not \subseteq S$}\newline Since $N[z_{i}] \subseteq
N[x_{i,j}]$, either $X_{i} \cap S = \varnothing$, or one of the $x_{i,j}$
annihilates the private neighbourhood of $z_{i}$ in $S$.

\item \emph{If $S$ is a minimal dominating set, and $x_{i,1} \notin S$, then
$|S\cap V(\mathcal{Z}_{i})| = 2$.}\newline If $z_{i} \in S$, then by (i),
$X_{i} \cap S = \varnothing$. To dominate $Y_{i}$ minimally, exactly one
$y_{i,j} \in S$, and thus $|V(\mathcal{Z}_{i}) \cap S|=2$. Suppose instead
that $z_{i} \notin S$ and that $z_{i}$ is externally dominated. Then $|X_{i}
\cap S| \geq1$, say without loss of generality, $x_{i,2} \in S$. To dominate
$y_{i,1}$ minimally, some $y_{i,k}$ is also in $S$. Then $\mathcal{Z}_{i}$ is
dominated by $\{x_{i,2}, y_{i,k}\}$ and again $|V(\mathcal{Z}_{i}) \cap S|=2$.

\item \emph{If $S$ is a minimal dominating set, then for $1\leq i\leq n$,
$|S\cap V(\mathcal{Z}_{i})|\leq3$, and $|S\cap V(\mathcal{Z}^{\ast})|\leq3$}.
\newline As in (ii), if some $x_{i,j}\in S$ and some $y_{i,k}\in S$, then
$|V(\mathcal{Z}_{i})\cap S|=2$. Thus, the largest intersection of
$\mathcal{Z}_{i}$ and $S$ occurs when $z_{i}\notin S$ and $S\cap
Y_{i}=\varnothing$; that is, when $V(\mathcal{Z}_{i})\cap S=X_{i}$.

\item \emph{If $S$ is a $\Gamma$-set of $G$, then $|V(H)\cap S|\leq1$%
.}\newline Suppose to the contrary that $|V(H)\cap S|=m\geq2$; say without
loss of generality that $v_{1},...,v_{m}\in S$. For each $i=1,...,m$,
$z^{\ast}\notin pn[v_{i},S]$. Since $z_{i}$ is dominated (by a vertex in
$\{z_{i},x_{i,1},x_{i,2},x_{i,3}\}\cap S$), $x_{i,1}\notin pn[v_{i},S]$. Hence
either $v_{i}\in pn[v_{i},S]$ or $v_{j}\in pn[v_{i},S]$ for some $j>m$. In the
former case, $x_{i,1}\notin S$ and $|V(\mathcal{Z}_{i})\cap S|=2$, and in the
latter case, $x_{j,1}\notin S$ and $|V(\mathcal{Z}_{j})\cap S|=2$. Thus, for
each $i\in\{1,...,m\}$ there exists a unique $j$ such that $x_{j,1}\notin S$.
By (ii), then, for each $i\in\{1,...,m\}$ there exists a unique $j$ such that
$|S\cap V(\mathcal{Z}_{j})|=2$. Hence $|S\cap(V(H)\cup(\bigcup_{i=1}%
^{n}\mathcal{Z}_{i}))|\leq3n$. By (iii), $|S\cap V(\mathcal{Z}^{\ast})|\leq3$.
Hence $|S|\leq3n+3<\Gamma(G)$, a contradiction.
\end{enumerate}

\noindent From (i)-(iv), $\mathcal{S}$ consists of all the $\Gamma$-sets of
$G$. The proof proceeds as in Theorem \ref{thm:gamma-ID}. To construct other
graphs with a $\Gamma$-graph of $H$, attach additional copies of $\mathcal{Z}$
to any vertex of $V(H)$.
\end{proof}



\section{Open problems}

We concluded with a few open problems for future consideration.

\begin{enumerate}
\setlength{\parindent}{1pt} \setlength{\parskip}{1pt} 

\item Determine conditions on the graph $G$ under which each $\gamma$-graph
variation is connected/disconnected. 

\item \emph{Reconfiguration problems} are a well-studied class of problems
which examine the step-by-step transformation from one feasible solution to
another, where feasibility is maintained at each intermediate step (see
\cite{IDHPSUU11}). These problems are often represented as
\emph{reconfiguration graphs}, where each vertex represents a feasible
solution. As they represent the movement from one $\gamma$-set to another with
each intermediate step also being a $\gamma$-set, $\gamma$-graphs and the
variations presented in this paper are reconfiguration graphs. In the context
of graph problems where the vertices of the reconfiguration graph
$\mathcal{G}$ represent vertex subsets of a graph $G$, a vertex $v\in V(G)$ in
a solution set $S$ is said to be \emph{stuck} if in $\mathcal{G}$, each
neighbor $v_{S^{\prime}}$ of the vertex $v_{S}$ corresponding to $S$ in $G$
has $v\in S^{\prime}$. A vertex $v\in S$ is \emph{frozen} if all vertices in
the same component of $\mathcal{G}$ as $v_{S}$ correspond to sets also
containing $v$. Under what conditions is a vertex stuck or frozen in each of
the $\gamma$-graph variations? Moreover, when is $v$ in every $\gamma$-graph
variation?

\item Let $\pi$ be any of the above-mentioned domination-related parameters.
Is it true that every \textbf{bipartite} graph is the $\pi$-graph of a
\textbf{bipartite }graph?

\item Study the nature of $i$-graphs, $\operatorname{IR}$-graphs, and $\alpha
$-graphs, where $\operatorname{IR}(G)$ is the upper irredundance number and
$i(G)$ and $\alpha(G)$ are the independent domination and independence numbers
of $G$, respectively. (Depending on $H$, the graph $G$ constructed in the
proof of Theorem \ref{thm:bigGamma} could have more $\operatorname{IR}$-sets
than the order of $H$, and possibly also $\operatorname{IR}(G)>\Gamma(G)$.) In
particular, determine whether all graphs are $i$-graphs, $\alpha$-graphs, or
$\operatorname{IR}$-graphs. Consider other domination variations like Roman
and Italian domination.
\end{enumerate}

\subsection*{Acknowledgement}

The authors wish to thank Rick Brewster for many helpful comments and suggestions.


\end{document}